\documentclass[12pt,a4paper,twoside]{article}
\usepackage[T1]{fontenc}
\usepackage[english]{babel}
\usepackage{fancyhdr}
\usepackage{indentfirst}
\usepackage{newlfont}
\usepackage{verbatim}
\usepackage{enumerate}
\usepackage{caption}
\usepackage{hyperref}
\usepackage{pdflscape}
\usepackage[all,cmtip]{xy}
\usepackage{amsthm}
\usepackage{amssymb}
\usepackage{amscd}
\usepackage{amsmath}
\usepackage{latexsym}
\usepackage{amsfonts}
\usepackage{amsbsy}
\usepackage{mathrsfs}
\usepackage[usenames,dvipsnames]{color}
\usepackage{graphicx}
\usepackage{color}

\usepackage{mathtools} 
\usepackage{adjustbox} 
\usepackage[percent]{overpic} 
\usepackage{subcaption} 

\theoremstyle{plain}
\newtheorem{teo}{Theorem}
\newtheorem{prop}{Proposition}

\theoremstyle{definition}

\newtheorem{lemma}{Lemma}

\newtheorem{rem}{Remark}

\def\B{\textup{B}}

\newcommand{\R}{\ensuremath{\mathbb R}}

\theoremstyle{remark}

\newcommand{\lra}{\longrightarrow}

\newcommand{\pmo}{^{\pm 1}}
\newcommand{\mo}{^{-1}}
\newcommand{\ms}[1]{{\small$#1$}} 
\newcommand{\mss}[1]{{\footnotesize$#1$}} 
\newcommand{\msss}[1]{{\scriptsize$#1$}} 
\newcommand{\mm}{\text{-}}
\newcommand{\pp}{\text{+}}

\begin{document}

\selectlanguage{english}

\title{A Markov theorem for generalized plat decomposition}

\author{Alessia~Cattabriga\footnote{A.~Cattabriga has been supported by the "National Group for Algebraic and Geometric Structures, and their Applications" (GNSAGA-INdAM) and University of Bologna, funds for selected research topics.} - Bostjan~Gabrov\v{s}ek\footnote{B.~Gabrov\v sek has been supported by Slovenian Research Agency grants J1-8131, J1-7025, and N1-0064.}}

\maketitle

\begin{abstract}
We prove a Markov theorem for tame links in a connected closed orientable 3-manifold $M$  with respect to a plat-like representation.  More precisely, given a genus $g$ Heegaard surface $\Sigma_g$ for $M$   we represent each link in $M$   as  the plat closure of a braid in the surface braid group $B_{g,2n}=\pi_1(C_{2n}(\Sigma_g))$ and analyze how to translate the  equivalence of links in $M$ under ambient isotopy into an algebraic equivalence in $B_{g,2n}$.  First, we study the equivalence problem  in $\Sigma_g\times [0,1]$,  and then, to obtain the equivalence in $M$, we investigate how   isotopies corresponding to ``sliding'' along meridian discs change the braid representative.  At the end we provide explicit constructions for Heegaard genus 1 manifolds, i.e. lens spaces and $S^2\times S^1$.\\

\noindent{{\it Mathematics Subject
Classification 2010:} Primary 57M27, 20F38; Secondary 57M25.}

\end{abstract}

\begin{section}{Introduction and preliminaries}
\label{preliminari}

The connection among links and braids dates back to the thirties when  the results of Alexander (\cite{A}) and Markov (\cite{Ma}) showed that it is possible to represent each link using a braid, by ``closing it up'',  and described the equivalence moves connecting two different braids representing the same link. Forty years later in \cite{B} Joan Birman investigated  another way to use braids with an even number of strands to represent links:   closing them  in the so called  ``plat'' way. She proved that each link is the plat closure of a braid  and that two different braids representing the same link are connected through a stabilization move and moves corresponding to the generators of  a subgroup of the braid group studied by Hilden in \cite{H}. 

After these pioneer works many important results succeeded one another, as for example the use of braid representations to construct links invariants. In this direction an interesting result is a   description of the Jones polynomial of a link  in terms of the action of a braid,  having the link as  plat closure, over a homological pairing  defined on a covering of the configuration space of $n$ points into the $2n$-punctured disc (see \cite{Big}).

In the light of these fruitful connections, many authors investigated the possibility to use braids to represent links also in  connected closed orientable 3-manifolds different from  $S^3$. A first generalization of Markov's and Alexander's results  was presented in \cite{Sk}, using the idea of fibered knots, while in \cite{DLP, Lamb1, Lamb2} another generalization is reached via  mixed braids, i.e., braids having a part of the  strands representing the ambient 3-manifold, via Dehn surgery. 

With respect to plat closure, the first attempt to generalize Birman's and Hilden's results was done in \cite{BC}, using the notion of generalized bridge decomposition, i.e.,   plat closure with respect to Heegaard surfaces. More precisely, the authors proved that given a connected closed orientable 3-manifold $M$ and  a Heegaard surface $\Sigma$ in $M$, each link  can be represented as the plat closure of an element  of $B_{2n}(\Sigma)$, the braid group on $2n$ strings of $\Sigma$. Moreover, they studied a subgroup  of $B_{2n}(\Sigma)$, that they named the generalized Hilden group, acting trivially in the representation and asked whether, as in the classical case, this group, up to  a stabilization move, is enough to describe the equivalence. 

In this paper we refer to this question by generalizing Birman's equivalence theorem in this setting, i.e., presenting a finite set of moves connecting  two braids in  $B_{2n}(\Sigma)$ representing isotopic links in the manifold (see Theorem~\ref{thm:markov}). The result is reached in two steps: first we study the equivalence problem in $\Sigma\times I$, then we add ``slide like'' moves to take care of isotopies that are defined in  the whole  manifold. While the moves arising in the first step do depend only on the genus of $\Sigma$, and so are the same in each manifold having Heegaard genus at most $g$, the nature of both the manifold and  the Heegaard surface involved in the representation are encoded in the slide like moves. In order to represent the moves geometrically we borrow the idea of arrow diagrams used in \cite{GM1, Mr1, Mr2}.

Our approach seems to be quite flexible, since, once the manifold and the Heegaard surface are fixed,  in order to describe explicitly the equivalence, it is enough to find the words in $B_{2n}(\Sigma)$ whose plat closures are the boundaries of two systems of  meridian disks of the Heegaard decomposition associated to $\Sigma$.

Further development on the topic may include studying the connection among surface braids when the Heegaard splitting is not fixed, e.g., studying the algebraic relation between braid representatives of a link with respect to the stabilization of Heegaard splittings. Another open question is the construction of new link invariants or the revision of old ones in terms of the braid representative. Finally, it could also be interesting to expand the set of examples, worked out in this paper for Heegaard genus one manifolds, to higher Heegaard genus manifolds.

The paper is organized as follows: in Section~\ref{Heegaard} we review the notion of generalized plat decomposition and how a link can be represented through elements of surface braid groups; in Section~\ref{Markov} we investigate the equivalence in thickened surfaces while in Section~\ref{slide} we prove the main theorem of the paper, that is, we describe the moves connecting braids representing isotopic links. The last section is devoted to the explicit algebraic description of the equivalence moves in manifolds with Heegaard genus at most one:   $S^3$, where  we obtain Birman's result, lens spaces and $S^2\times S^1$. \\

We end this section by recalling some well-known facts  about link isotopy in 3-manifolds in order to fix our notations and conventions.  Throughout the text all the 3-manifolds are supposed to be connected, closed and orientable and we will consider only tame links (i.e., 1-dimensional closed PL-submanifolds). 

Two links $L$ and $L'$  in a 3-manifold $M$ are \emph{equivalent} if there exists an ambient isotopy of $M$ taking $L$ into $L'$, i.e., there exist a PL-map $H:M\times [0,1]\to M$ such that: $H(\cdot,t)$ is a PL-homeomorphism of $M$ for each $t \in [0,1]$,  $H(\cdot, 0)=\textup{id}_M$ and $H(L,1)=L'$. 

In the PL-category ambient isotopy is realized through a finite sequence of the so called $\Delta$-moves. A $\Delta$\emph{-move} (and its inverse) on a link $L$ in $M$ is a local elementary combinatorial isotopy move, realized as follows: (1) take a closed ball $B^3$ embedded into $M$ and such that $L \cap B^3$ is a trivial arc $a$  properly embedded in $B^3$ (i.e., an arc that co-bounds an embedded disk with another arc on $\partial B^3$);  (2)  replace the arc $a$ by two other arcs such that all three arcs span an embedded triangle in $B^3$ intersecting $L$ only in $a$. If  $[V_1,\ldots, V_n]$ denotes the $(n-1)$-simplex having $V_1,\ldots, V_n$ as vertices, the previous move can be combinatorially described as
$$L\longleftrightarrow (L-[P,Q])\cup [P,R]\cup[R,Q]$$
with $\{P,Q\}=\partial a$ and $[P,Q,R]\cap L=[P,Q]$.

\medskip

\medskip

\noindent \textit{Acknowledgement:} A.~Cattabriga  has been supported by the ``National Group for Algebraic and Geometric Structures, and their Applications'' (GNSAGA-INdAM) and University of Bologna. B.~Gabrov\v{s}ek was financially supported by the Slovenian Research Agency grants J1-8131, J1-7025, and N1-0064. 

\end{section}

\begin{section}{Heegaard surfaces and generalized plat decompositions}

\label{Heegaard}
In this section  we review the notion of generalized plat decomposition, introduced in \cite{Do}, and describe how a link can be represented through elements of surface braid groups (see \cite{BC}). We end the section by describing the set of generators of surface braid groups given in \cite{Be}.\\

A \emph{Heegaard surface} for a 3-manifold $M$ is a connected closed orientable surface $\Sigma$  embedded  in $M$ such that $M\setminus \Sigma$ is the disjoint union of two handlebodies (of the same genus).  From an extrinsic point of view, we can say that  $M$ is homeomorphic to $H_1 \cup_h H_2$, where $H_1$ and $H_2$ are two oriented copies of a standard handlebody in $\mathbb R^3$ (see Figure~\ref{figstandard})   and $h:\partial H_2 \lra \partial H_1 $ is an orientation reversing homeomorphism. The triple $(H_1,H_2,h)$ is called \emph{Heegaard splitting} of $M$ and the Heegaard surface is $\partial H_1\cup_h\partial H_2$.   Each $3$-manifold admits Heegaard splittings  (see \cite{He}), moreover,  Heegaard splittings are 3-dimensional cases of symmetric version handle decomposition, that holds for each  differentiable manifolds of odd dimension (see \cite{Mi}).  Indeed,  one handlebody is obtained by attaching  $g$ $1$-handles to a $0$-handle, while attaching  $g$ $2$-handles to one $3$-handle, gives, up to duality, the other handlebody. The \emph{Heegaard genus} of a 3-manifold $M$ is the minimal genus of a Heegaard surface for $M$: the $3$-sphere $S^3$ is the only $3$-manifold with Heegaard genus $0$, while the manifolds with Heegaard genus $1$ are lens spaces (i.e., cyclic quotients of $S^3$) and $S^2\times S^1$. 
While $S^3$, as well as lens spaces, have, up to isotopy, only one Heegaard surface of minimal genus (and those of higher genera are stabilizations of that of minimal genus), in general, a manifold may admit non isotopic  Heegaard surfaces of the same genus (see for example  \cite{M} for the case of Seifert manifolds).\\

Heegaard surfaces are the tool that leads to a generalization of the classical  notion of bridge decomposition for links in   $\mathbb R^3$ (or  $S^3$)  to the case of 3-manifolds (see \cite{Do}). 
Given a handlebody $H$, we say  that a  set of $n$ properly embedded disjoint arcs  $\{A_1,\ldots ,A_n\}$ is a \emph{trivial system} of arcs
if  there exist $n$ mutually disjoint embedded discs,
called \emph{trivializing discs}, $D_1,\ldots ,D_n\subset H$
such that $A_i\cap D_i= A_i\cap\partial D_i=A_i$, $A_i\cap
D_j=\emptyset$ and $\partial D_i-A_i\subset\partial H$ for all
$i,j=1,\ldots ,n$ and $i\neq j$. We say that the arc $A_i$ \emph{projects} onto the arc $\partial D_i-\textup{int}(A_i)\subset\partial H$ (via $D_i$). Let $\Sigma$ be a Heegaard surface for $M$.   We say that a link $L$ in $M$ is in \emph{bridge position} with respect to $\Sigma$ if:  
\begin{itemize}
\item[(i)] $L$ intersects $\Sigma$ transversally and
\item[(ii)]the intersection of $L$ with both handlebodies, obtained  splitting $M$ by  $\Sigma$,
 is a trivial system of arcs.
\end{itemize}
Such a decomposition for $L$ is called $(g,n)$\emph{-decomposition} or $n$-\emph{bridge
decomposition of genus} $g$, where  $g$ is the genus of $\Sigma$ and $n$ is the cardinality of the trivial system. The minimal $n$ such that
$L$ admits a $(g,n)$-decomposition is called the \emph{genus} $g$ \emph{bridge number} of $L$.
Clearly if $g=0$, the manifold $M$ is the 3-sphere and   we get the usual notion of bridge decomposition and
bridge number of links in the 3-sphere (or in $\R^3$).  The notion of $(g,b)$-decomposition is a useful tool 
to study links in $3$-manifolds (see for example \cite{C, CM2, CMV, GMM}).

As bridge decompositions of links in $S^3$ (or $\mathbb R^3$) are connected to plat closures of classical braids, $(g,b)$-decompositions can be used to represent links in $3$-manifolds via braid groups of surfaces as follows (see also \cite{BC,CM}).

Let  $\Sigma_g$ be a genus $g$ Heegaard surface for a $3$-manifold $M$ and  let $ \mathbf c =\{ c_1,\ldots, c_g\}$ and  $ \mathbf c^* =\{ c^*_1,\ldots, c^*_g\}$ be the boundaries of two systems of  meridian discs of the two handlebodies.\footnote{Recall that  a genus $g$  handlebody is uniquely determined, by the boundary surface and  the boundaries of a system of $g$ disjoint meridian discs whose complement is (homeomorphic to) the 3-ball.} In terms of handle decomposition we can think of $\mathbf c$ as the \emph{attaching circles} of the $2$-handles and of $\mathbf c^*$ as the \emph{dual attaching circles} of the $1$-handles, i.e., the attaching circles of the dual 2-handles in the ``upside down'' decomposition (see \cite{GS}).
Starting from the data $(\Sigma_g,\mathbf c, \mathbf c^*)$ we can reconstruct $M$ by: considering the thickened surface $\Sigma_g \times [0,1]$, gluing  $2$-handles  along the curves   $ \mathbf c\times \{1\} \subset \Sigma_g\times \{1\}$ and  another set of (dual) $2$-handles along the curves $\mathbf c^*\times \{0\} \subset\Sigma_g\times \{0\}$, and closing the resulting manifold by attaching a $3$-handle and a (dual) $3$-handle. Given a Heegaard splitting of a 3-manifold $M$, we call the triple $(\Sigma_g,\mathbf c, \mathbf c^*)$ \textit{Heegaard diagram} of the splitting. 
Up to homeomorphism, we can always suppose that $\Sigma_g$ and  $ \mathbf c^*$ are as depicted in Figure~\ref{figstandard}, while $\mathbf c$  depends on $M$ and on the chosen Heegaard surface. 

Referring to  Figure~\ref{figstandard}, let  $\mathcal P_{2n}=\{P_1,\ldots,P_{2n}\}$ be a set of $2n$ distinct points  on $\Sigma_g$ and denote with  $B_{g,2n}$, the braid group  on $2n$-strands of $\Sigma_g$, i.e, the fundamental group  of the configuration space of the $2n$ points in $\Sigma_g$.  Fix a set of $n$ arcs $\gamma_1,\ldots,\gamma_n$ embedded into $\Sigma_g$, such that $\gamma_i\cap \gamma_j=\emptyset$ if $i\ne j$ and $\partial \gamma_i=\{P_{2i-1},P_{2i}\}$, for $i,j=1,\ldots,n$.  Given an element $\beta\in B_{g,2n}$,  realize it as a geometric braid, that is, as a set of $2n$ disjoint paths in $\Sigma_g\times\left[ 0,1\right]$ connecting   $\mathcal P_{2n}\times \left\{0\right\}$ to 
 $\mathcal P_{2n}\times \left\{1\right\}$.  The  \emph{plat closure} $\widehat{\beta}\subset M$ of $\beta$  is  the link obtained  ``closing'' $\beta$ by connecting   $P_{2i-1}\times\{0\}$ with  $P_{2i}\times\{0\}$ through   $\gamma_i\times\{0\}$ and $P_{2i-1}\times\{1\}$ with  $P_{2i}\times\{1\}$  through  $\gamma_i\times\{1\}$, for $i=1,\ldots,n$. Clearly,  $\widehat{\beta}$ is in bridge position with respect to $\Sigma_g$ and so it has genus $g$ bridge number at most $n$. 
 Note that this closing procedure does not depend on the system of
 meridian curves $\mathbf c$, however, as we will see, $\mathbf c$ characterizes the manifold, so
 the isotopy type of the resulting link does depend on it.
 
In Figure~\ref{figex} an example is represented with $g=1$ and $n=2$, with $\Sigma_1\cong S^1\times S^1$ represented as a square with opposite sides identified.

\begin{figure}[ht]
\centering
	\begin{overpic}[page=38]{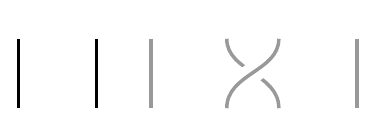}%
	 \put(9,5.5){$c^*_1$}\put(24,5.5){$c^*_2$}\put(38,5.5){$\cdots$}\put(46.5,5.5){$c^*_g$}
	 \put(59,12){\mss{P_1}}\put(65,12){\mss{P_2}}
	 \put(71,12){\mss{P_3}}\put(77,12){\mss{P_4}}\put(81,9.75){\mss{\cdots}}
	 \put(84,12){\mss{P_{2n-1}}}\put(93,12){\mss{P_{2n}}}
	 \put(62.5,7){$\gamma_1$}\put(74,7){$\gamma_1$}\put(89.5,7){$\gamma_n$}
	 \end{overpic}
        \caption{The standard choice for $\Sigma_g$, $\mathbf c^* =\{c^*_1,\ldots,c^*_g\}$ and $\gamma_1,\ldots,\gamma_n$.}
\label{figstandard}
\end{figure}
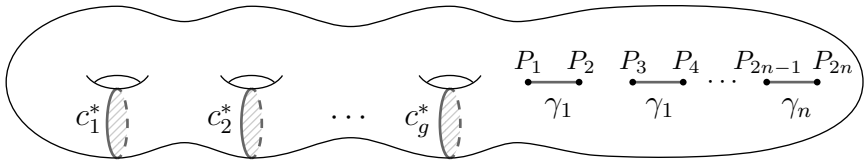

\begin{figure}[ht]
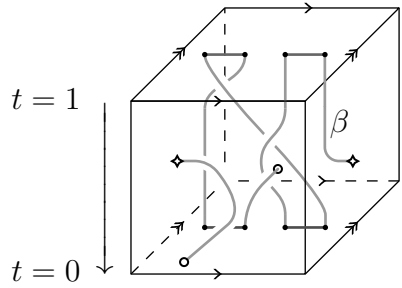

\centering
	\begin{overpic}[page=39]{images}%
	\put(72,53){$\beta$}
	\put(-7,64){\rotatebox{270}{$\xrightarrow{\makebox[2.1cm]{}}$}}
	\put(-36,2){$t=0$}\put(-36,61){$t=1$}
	 \end{overpic}
        \caption{An example of a closure $\widehat{\beta}$ with $\beta \in B_{1,4}$.}
\label{figex}
\end{figure}

As in the classical case,   it is possible to prove an Alexander representation theorem, i.e., each link in $M$ is  isotopic to  the plat closure  of a braid in $B_{g,2n}$. This follows essentially from \cite[Proposition 4.6]{BC}, where, however,  a slightly different approach is used. So  here we describe a more topological proof using techniques analogous to those used in  \cite[Lemma 2]{B}, that will be useful throughout the rest of the text.

\begin{teo}\label{lemma:lins}
Every   link $L$ in a 3-manifold $M$ having  Heegaard  genus at most $g$, may be braided to  a geometric braid in $B_{g,2n}$, the plat closure of which is equivalent to $L$.
\end{teo}
\begin{proof}
Let $(\Sigma_g,\mathbf c,\mathbf c^*)$ be a Heegaard diagram for $M$, where $\Sigma_g$ and $\mathbf c^*$ are those depicted in  Figure~\ref{figstandard}. Up to isotopy, we can assume that $L\subset N(\Sigma_g)$, where $N(\Sigma_g)$ denotes a closed tubular neighborhood of $\Sigma_g$. We choose a parametrization  $N(\Sigma_g)\cong \Sigma_g\times I$, with $I=[0,1]$, such that $L\subset \Sigma_g\times[0.25,0.75]$ and  fix an orientation on each component of $L$.
 By \emph{projecting} a point of $N(\Sigma_g)$ onto a boundary component of $\partial(\Sigma_g\times I)$, we mean moving the point along the $I$ trivial fibration until the required boundary component is reached. Let $Q_1,\ldots, Q_k$, be  the vertices in a PL-decomposition of  $L$  having  the following property: 
 
 \begin{alignat*}{1} 
\ & \textup{given a triple of consecutive points (according to the} \notag \\ \ & \textup{orientation of } L\textup{) } Q_{i-1},Q_i,Q_{i+1}\textup{ there exist an  open  } \notag \\
\ & \textup{neighborhood } N_i \textup { of the } I \textup{ fibre trough } Q_i \textup{ in }N(\Sigma_g) \textup{ whose  } \tag{*}\\\ & \textup{intersection with } L \textup{ is contained in  } [Q_{i-1}Q_i]\cup  [Q_iQ_{i+1}].
 \end{alignat*}

\begin{figure}[ht]
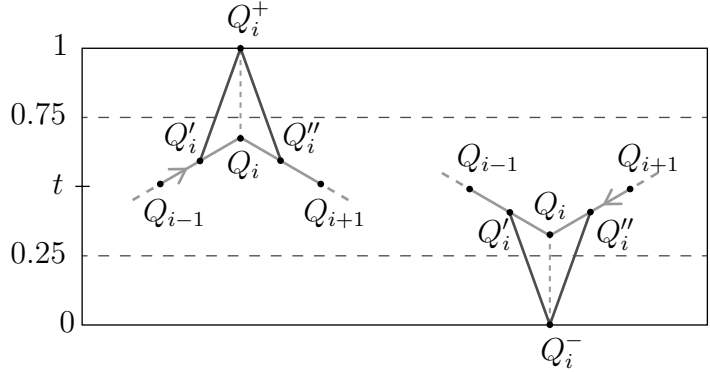

\centering
	\begin{overpic}[page=71]{images}
		\put(-3,34){$0.75$}\put(-3,14){$0.25$}\put(3,24){$t$}
		\put(3,44){$1$}\put(4,4){$0$}
		\put(16,20){$Q_{i-1}$}\put(28,27){$Q_i$}\put(39,20){$Q_{i+1}$}
		\put(19,31){$Q'_i$}\put(28,48){$Q^+_i$}\put(36,31){$Q''_i$}
		\put(61,28){$Q_{i-1}$}\put(72,21){$Q_i$}\put(84,28){$Q_{i+1}$}
		\put(64,17){$Q'_i$}\put(73,0){$Q^-_i$}\put(81,17){$Q''_i$}
	 \end{overpic}
        \caption{The braiding process.}
\label{fig_alex}
\end{figure}

 
 By a general position argument we can assume that $L$, up to isotopy,  contains no arcs  which are  horizontal with respect to the height function associated to the projection onto the $I$ factor. Moreover, we say that a subarc of $L$ is oriented upwards (resp. downwards) if moving along it, with respect to the fixed orientation, its projection over $I$ is increasing (resp. decreasing). Note that each component of $L$ contains at least one arc oriented upwards and one arc oriented downwards.

 We will construct a link $L'$ isotopic to $L$ such that:
 \begin{itemize}
  \item[(1)] $L'$ is contained in $N(\Sigma_g)$,
  \item[(2)] the link $L'$ meets both $\Sigma_g\times \{0\}$ and $\Sigma_g\times \{1\}$ in $n$ points,  and meets $\Sigma_g\times \{t\}$, for each $t\in(0,1)$ in exactly $2n$ points.  
 \end{itemize}
 For each connected component of $L$, fix a point inside it contained in an arc oriented upwards and repeat the following process:  moving along the  component according to the fixed orientation and starting from the distinguished point let $[Q_{i},Q_{i+1}]$ be the first arc which is oriented downwards;  consider the neighborhood $N_i$ of $Q_i\times I$ defined in (*) and let $Q'_i$, $Q_i''$ be the unique point of intersection of $\partial N_i$ with, respectively, $[Q_{i-1},Q_i]$ and $[Q_{i},Q_{i+1}]$ and $Q_i^+$ be the projection of $Q_i$ onto $\Sigma_g\times\{1\}$  (see Figure~\ref{fig_alex}). Replace  $[Q_{i-1},Q_i]\cup [Q_{i},Q_{i+1}]$ with $[Q_{i-1},Q_i']\cup[Q_i',Q_i^+]\cup[Q^+_{i},Q''_{i}]\cup [Q''_i,Q_{i+1}]$. This move can be clearly decomposed into  a sequence of $\Delta$-moves.  Now go on, along the same component of the link, until you meet the first subarc $[Q_{k}, Q_{k+1}]$ which is oriented upwards  and replace    $[Q_{k-1},Q_k]\cup [Q_{k},Q_{k+1}]$ with $[Q_{k-1},Q_k']\cup[Q_k',Q_k^-]\cup[Q^-_{k},Q''_{k}]\cup [Q''_k,Q_{k+1}]$, where now $Q^-_k$ is the projection of $Q_k$ onto $\Sigma_g\times \{0\}$. Let $L'$ be the link obtained at the end of the process: clearly $L'$ is isotopic to $L$ and $L'$ satisfies properties (1) and (2).
 
 Let $\mathcal Q^-=\{Q_1^-,\ldots, Q_n^-\}$ (resp. $\mathcal Q^+=\{Q_1^+,\ldots, Q_n^+\}$) be the set of points in $\Sigma_g$ defined by the condition $L'\cap \left(\Sigma_g\times \{0\}\right)=\mathcal Q^-\times \{0\}$ (resp.   $L'\cap \left(\Sigma_g\times \{1\}\right)=\mathcal Q^+\times \{1\}$). 
 Referring to Figure~\ref{figstandard}, for each arc $\gamma_i$,  let $B_i$ be an internal point of the arc. Since the configuration  space $C_n(\Sigma_g)$ of $n$ point in $\Sigma_g$ is arc connected (see \cite{B1}) there exist two paths $p^-(t):I\to C_n(\Sigma_g)$ and $p^+(t):I\to C_n(\Sigma_g)$ such that $p^-(0)=\{B_1,\ldots, B_n\}=p^+(1)$, $p^-(1)=\mathcal Q^-$ and $p^+(0)=\mathcal Q^+$. Properly rescaling the interval $I$ and deforming $L'$ along the  graph of such paths, we obtain an isotopic  link $L''$  satisfying (1) and (2) and such that $L''\cap \left(\Sigma_g\times \{0\}\right)=\{B_1,\ldots, B_n\}\times\{0\}$ and $L''\cap\left(\Sigma_g\times \{1\}\right)=\{B_1,\ldots, B_n\}\times\{1\}$. Then there clearly exists an $\eta>0$ such that $L''\cap [\eta,1-\eta]$ is a well defined element $\beta\in B_{g,2n}$, with $\widehat\beta$ equivalent to $L''$.
\end{proof}

\begin{rem} \label{thick} The procedure described in the previous proof is called the \emph{braiding process}. Since the braiding process is realized in $\Sigma_g\times I$, we have that also  each link in a thickened surface is equivalent to the plat closure of a braid.   
\end{rem}

We end this section by recalling the presentation of   $B_{g,2n}$  given in \cite{Be}. The generators are:   $\sigma_1,\ldots,\sigma_{2n-1}$, the standard braid ones, and $a_1,\ldots,a_g,b_1,\ldots, b_g$, where $a_i$ (resp. $b_i$) is the braid whose strands are all trivial except the first one which goes once along the $i$-th longitude (resp. $i$-th meridian) of $\Sigma_g$ (see Figure~\ref{figgen}).  The relations are the following ones:

\begin{figure}[h!]
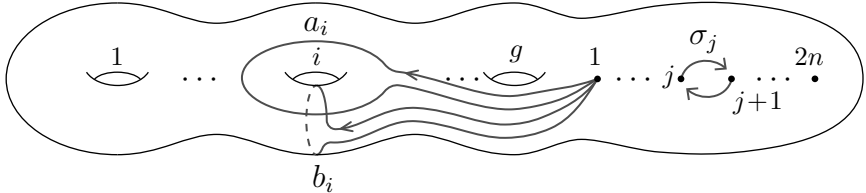

\centering
\begin{overpic}[page=64]{images}\put(35,16){$a_i$}\put(36,-2){$b_i$}\put(79,14.5){$\sigma_j$}
\put(21,9.5){$\cdots$}\put(51,9.5){$\cdots$}\put(70.5,9.5){$\cdots$}\put(86.5,9.5){$\cdots$}
\put(13,12){\mss{1}}\put(36,12){\mss{i}}\put(58.5,13){\mss{g}}
\put(67.5,12){\mss{1}}\put(76,10){\mss{j}}\put(84,7){\mss{j\!+\!1}}\put(91,12){\mss{2n}}
	 \end{overpic}
\caption{The generators of $B_{g,2n}$.}
\label{figgen}
\end{figure}


 \begin{itemize}
 \item[-] Braid relations:
                  $$\sigma_i\sigma_{i+1}\sigma_i=\sigma_{i+1}\sigma_i\sigma_{i+1}\qquad(i=1,\ldots,2n-2)$$
                  $$\sigma_{i}\sigma_{j}=\sigma_{j}\sigma_{i}\qquad(i,j=1,\ldots,2n-1,\ \vert i-j\vert \geq 2)$$
  \item[-] Mixed relations:
\begin{align*}
  (R1)  & \qquad a_r\sigma_i=\sigma_i a_r \qquad(1\leq r\leq g,\ i\ne 1)\\
  \ & \qquad b_r\sigma_i=\sigma_i b_r \qquad(1\leq r\leq g,\ i\ne 1)   \\
  (R2)  & \qquad \sigma_1^{-1} a_r\sigma_1^{-1}a_r =a_r\sigma_1^{-1}a_r \sigma_1 ^{-1}\qquad(1\leq r\leq g)\\
  \ & \qquad \sigma_1^{-1} b_r\sigma_1^{-1}b_r =b_r\sigma_1^{-1}b_r \sigma_1 ^{-1}\qquad(1\leq r\leq g)  \\
    (R3)  & \qquad \sigma_1^{-1} a_s\sigma_1 a_r =a_r\sigma_1^{-1}a_s\sigma_1 \qquad(s<r)\\
  \ & \qquad \sigma_1^{-1} b_s\sigma_1 b_r =b_r\sigma_1^{-1}b_s \sigma_1\qquad(s<r)  \\
  \ & \qquad \sigma_1^{-1} a_s\sigma_1 b_r =b_r\sigma_1^{-1}a_s \sigma_1\qquad(s<r)  \\
  \ & \qquad \sigma_1^{-1} b_s\sigma_1 a_r =a_r\sigma_1^{-1}b_s \sigma_1\qquad(s<r)  \\
  (R4)  & \qquad \sigma_1^{-1} a_r\sigma_1^{-1}b_r =b_r\sigma_1^{-1}a_r\sigma_1 \qquad(1\leq r\leq g)\\
  (TR) &\qquad \left[a_1,b_1^{-1}\right]\cdots  \left[a_g,b_g^{-1}\right]=\sigma_1\sigma_2\cdots\sigma_{2n-1}^2\cdots \sigma_2\sigma_1
  \end{align*}
 \end{itemize}
where $[a,b]:=aba^{-1}b^{-1}$. 

We will depict a braid in $B_{g,2n}$ 
using a set of $g$ \emph{fixed strands} on the left (that intuitively represent the $g$ holes of $\Sigma_g$) and $2n$ \emph{moving strands} on the
right, which represent the braid. As depicted in Figure~\ref{fignotation}, 
we represent the generator $a_i$ and its inverse by the first moving strand winding around the $i$-th fixed strand and going over the rest of the fixed strands on the right, for $i=1,\ldots, g$. We represent $b_i$ (resp. $b_i^{-1}$), $i = 1,\ldots,g$, by an arrow labelled $i$ on the first moving strand pointing downwards (resp. upwards).


\begin{figure}[h!]
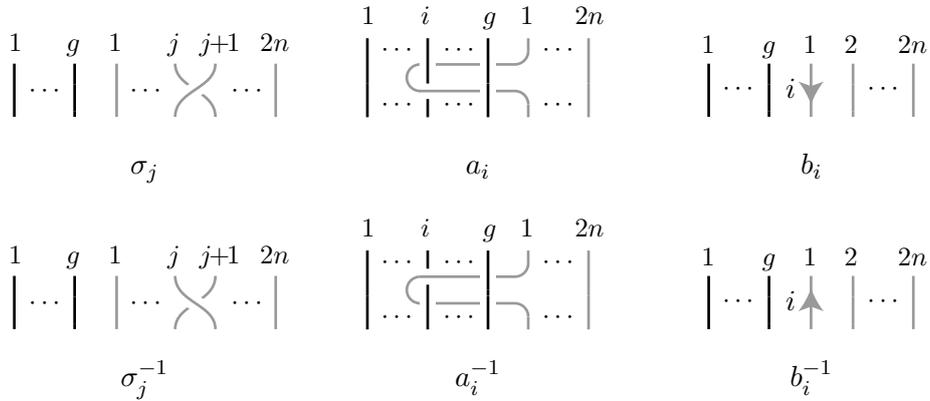

\centering

\begin{subfigure}[b]{.3\linewidth}\centering
\begin{overpic}[page=1]{images}\put(3,26){\mss{1}}\put(23,26){\mss{g}}
\put(37.5,26){\mss{1}}\put(58,26){\mss{j}}\put(69,26){\mss{j\!\!+\!\!1}}\put(90,26){\mss{2n}}
\put(10,10){\mss{\cdots}}\put(45,10){\mss{\cdots}}\put(80,10){\mss{\cdots}}
\end{overpic}\caption*{$\sigma_j$}\end{subfigure}
\
\begin{subfigure}[b]{.3\linewidth}\centering
\begin{overpic}[page=3]{images}\put(3,41){\mss{1}}\put(27,41){\mss{i}}\put(52,41){\mss{g}}
\put(67,41){\mss{1}}\put(89,41){\mss{2n}}
\put(11,28){\mss{\cdots}}\put(36,28){\mss{\cdots}}\put(76,28){\mss{\cdots}}
\put(11,6){\mss{\cdots}}\put(36,6){\mss{\cdots}}\put(76,6){\mss{\cdots}}
\end{overpic}\caption*{$a_i$}\end{subfigure}
\
\begin{subfigure}[b]{.3\linewidth}\centering
\begin{overpic}[page=6]{images}\put(3,31){\mss{1}}\put(29,31){\mss{g}}
\put(46.5,31){\mss{1}}\put(64,31){\mss{2}}\put(87,31){\mss{2n}}
\put(12,13){\mss{\cdots}}\put(74,13){\mss{\cdots}}\put(39,12){\ms{i}}
\end{overpic}\caption*{$b_i$}\end{subfigure}
\\[1em] 
\begin{subfigure}[b]{.3\linewidth}\centering
\begin{overpic}[page=2]{images}\put(3,26){\mss{1}}\put(23,26){\mss{g}}
\put(37.5,26){\mss{1}}\put(58,26){\mss{j}}\put(69,26){\mss{j\!\!+\!\!1}}\put(90,26){\mss{2n}}
\put(10,10){\mss{\cdots}}\put(45,10){\mss{\cdots}}\put(80,10){\mss{\cdots}}
\end{overpic}\caption*{$\sigma_j^{-1}$}\end{subfigure}
\
\begin{subfigure}[b]{.3\linewidth}\centering
\begin{overpic}[page=4]{images}\put(3,41){\mss{1}}\put(27,41){\mss{i}}\put(52,41){\mss{g}}
\put(67,41){\mss{1}}\put(89,41){\mss{2n}}
\put(11,28){\mss{\cdots}}\put(36,28){\mss{\cdots}}\put(76,28){\mss{\cdots}}
\put(11,6){\mss{\cdots}}\put(36,6){\mss{\cdots}}\put(76,6){\mss{\cdots}}
\end{overpic}\caption*{$a_i^{-1}$}\end{subfigure}
\
\begin{subfigure}[b]{.3\linewidth}\centering
\begin{overpic}[page=5]{images}\put(3,31){\mss{1}}\put(29,31){\mss{g}}
\put(46.5,31){\mss{1}}\put(64,31){\mss{2}}\put(87,31){\mss{2n}}
\put(12,13){\mss{\cdots}}\put(74,13){\mss{\cdots}}\put(39,12){\ms{i}}
\end{overpic}\caption*{$b_i^{-1}$}\end{subfigure}
\caption{Representing the generators of $B_{g,2n}$ and their inverses.}
\label{fignotation}
\end{figure}


\end{section}

\begin{section}{Markov theorem in thickened surfaces}
\label{Markov}

In this section we study the combinatorial equivalence for links in a thickened surface, generalizing results of \cite{B}. Once the statement is established,    the proof given in \cite{B} works almost without changes  also in this more general setting. Nevertheless, for readers' convenience, we  report here the main steps of the proof, with the only exception of Lemma 8 of \cite{B}  (see Proposition~\ref{lemma_birman}), which is really technical. \\

Let $\Sigma_g$ be the genus $g$ surface depicted in Figure~\ref{figstandard}. A link $L\subset \Sigma_g\times I$ is said to be in \textit{standard  position} if it intersects both $\Sigma_g\times \{0\}$ and $\Sigma_g\times \{1\}$ in $n$ points,  and meets $\Sigma_g\times \{t\}$, for each $t\in(0,1)$ in exactly $2n$ points.  We call the points of $\Sigma_g\times \{1\}$ (resp. $\Sigma_g\times \{0\}$) \emph{upper} (resp. \emph{lower}) \emph{boundary} points. A \emph{boundary point} is either an upper or a lower boundary point, while an \emph{interior point} is a point in $\Sigma_g\times (0,1)$.

The proof of Theorem \ref{lemma:lins}, along with the Remark  \ref{thick}, shows that each link in  $\Sigma_g\times I$ is isotopic to a link in standard position. Moreover it shows that each link $L$  in standard position is equivalent  to the plat closure of a braid in $B_{g,2n}$, where $n$ is the number of upper (or lower) boundary points of $L$.



Given a link $L$ in standard position, fix an orientation on it and consider  the following two moves, introduced in \cite{B}:

\begin{itemize}
 \item \emph{the spike move:} referring to Figure~\ref{figspike}, let $Q$, $B$, $P$ be consecutive points on $L$ such that $Q,P$ are interior points,  $B$ is a boundary point and $[Q,B,P]$ is an (embedded) triangle in $\Sigma_g\times I$ such that $[Q,B,P]\cap L=[Q,B]\cup[B,P]$. Let $B'\ne B$ be another boundary point on the same boundary component as $B$ and let $Q'$, $P'$ be interior points  satisfying the conditions $[P,P',Q']\cap L=\{P\}$, 
 $[Q,P',Q']\cap L=\{Q\}$ and $[P',B',Q']\cap L=\emptyset$. We replace   $[Q,B]\cup [B,P]$ on $L$ with  $ [Q,Q'] \cup [Q', B']\cup [B',P']\cup [P',P] $. 
Note that the spike move is executed by retracting the ``spike'' at $B$ to the ``base'' $[P,Q]$ and shooting out a new spike at $B'$, which will in general thread in and out of the other arcs of the link. So, generally it will not be possible to get the same result by  moving  the spike at $B$ directly along the boundary surface toward $B'$, since the other arcs of $L$ may interfere with such a move. 
 
\begin{figure}[ht]
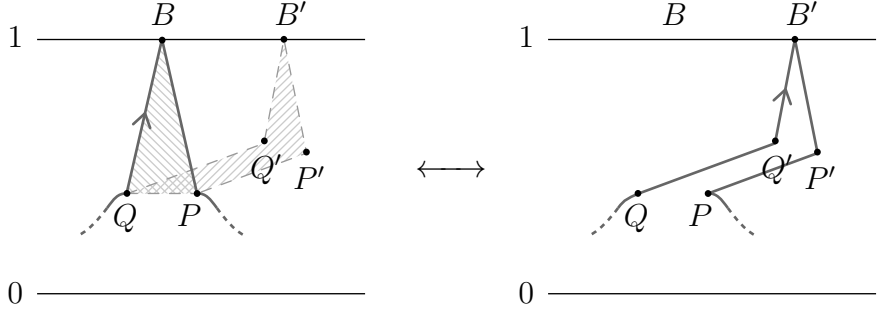

\centering
\begin{tabular}{ccc}
	\begin{overpic}[page=46]{images}%
	\put(0,70){$1$}\put(0,1.5){$0$}\put(38,76){$B$}\put(71,76){$B'$}
	\put(28,22){$Q$}\put(45,22){$P$}\put(64.5,35){$Q'$}\put(76,32){$P'$}
	\end{overpic} &
	\raisebox{1.75cm}{$\xleftrightarrow{\makebox[0.7cm]{}}$} &
	\begin{overpic}[page=47]{images}%
	\put(0,70){$1$}\put(0,1.5){$0$}\put(38,76){$B$}\put(71,76){$B'$}
	\put(28,22){$Q$}\put(45,22){$P$}\put(64.5,35){$Q'$}\put(76,32){$P'$}
	\end{overpic}
\end{tabular}
        \caption{The spike move on an upper boundary point.}
\label{figspike}
\end{figure}
 
 
 \item \emph{the stabilizing move:} referring   to Figure~\ref{figstab},  let $R$ be any internal point on $L$.  We can always assume, up to isotopy,  that the three consecutive points $P$, $R$, $Q$  on $L$  satisfy property (*). Let $R'$, $R''$ be the unique points of intersection of $\partial N$ (see  property (*)) with, respectively, $[P,R]$ and $[R,Q]$ and $R^+$,  $R^-$  be the projection of $R$ onto, respectively,  $\Sigma_g\times\{1\}$ and $\Sigma_g\times\{0\}$. We replace  $[R',R]\cup [R,R'']$ on $L$ with  $[R',R^+]\cup [R^+,R^-]\cup [R^-,R'']$. 
 
 \begin{figure}[ht]
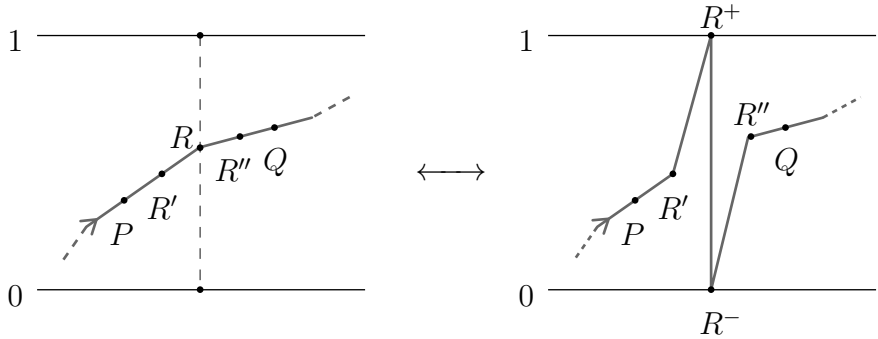

\centering
\begin{tabular}{ccc}
	\begin{overpic}[page=48]{images}%
	\put(0,70){$1$}\put(0,1.5){$0$}
	\put(27,18){$P$}\put(37,25){$R'$}\put(43,44){$R$}\put(55,35){$R''$}\put(68,38){$Q$}
	\end{overpic} &
	\raisebox{1.75cm}{$\xleftrightarrow{\makebox[0.7cm]{}}$} &
	\begin{overpic}[page=49]{images}%
	\put(0,70){$1$}\put(0,1.5){$0$}\put(48,76){$R^+$}\put(48,-6){$R^-$}
	\put(27,18){$P$}\put(37,25){$R'$}\put(57,50){$R''$}\put(68,38){$Q$}
	\end{overpic}
\end{tabular}
        \caption{The stabilization move.}
\label{figstab}
\end{figure}


\end{itemize}

Clearly both the moves  do not alter the property of being in standard position and are compositions of $\Delta$-moves. As a consequence, if we apply such  moves to a link in standard position, we will obtain an equivalent link still in standard position.  Moreover, while a spike move does not change the number of upper (or lower) boundary points, the stabilization move increases or decreases it by one.

In \cite[Lemma 8]{B} the converse statement is proved in the case of $g=0$, that is to say, for the classical plat closure in $ D^2\times I$  or $S^2\times I$. The proof  extends without changes to the case of higher genus,  so we get the following result.

\begin{lemma}

\label{lemma_birman}
Let $L$ and $L'$    be two  equivalent links in standard position. Then it is possible to connect  $L$ with $L'$ by a finite sequence of spike and stabilization moves.  
\end{lemma}

Consider the arcs $\gamma_1,\ldots,\gamma_n$ depicted in Figure~\ref{figstandard}. In order to get an algebraic equivalence among braids, we introduce some specific elements of $B_{g,2n}$ (see Figure~\ref{fig:f}):

\begin{itemize}
\item[-]  \emph{braid twists or intervals:} are braids exchanging the endpoints of an arc  $\gamma_i$;  in terms of the generators of $B_{g,2n}$ they are  the elements $\sigma_{2i-1}$, for $i=1,\ldots,n$; 

\item[-]  \emph{elementary exchanges of two arcs:} are braids exchanging two arcs $\gamma_i$ and $\gamma_{j}$; it is possible to write them as products of elementary  exchanges of neighborhood arcs, that is, as products of the elements $\sigma_{2i}\sigma_{2i+1}\sigma_{2i-1}\sigma_{2i}$, exchanging $\gamma_i$ and $\gamma_{i+1}$ for $i=1,\ldots, n-1$;

\item[-]  \emph{slides} of the  $i$-th arc: is a braid obtained by moving the  both the endpoints  $P_{2i-1}$ and $P_{2i}$ of an arc $\gamma_i$ along parallel paths; any slide of the $i$-th arc can be written as $c\sigma c^{-1}$, where $c$ is an elementary exchange taking the $i$-th arc into the first one and $\sigma$ is a slide of the first arc; moreover, a slide of the first arc can be written as a   product of  the following slides  
\begin{itemize}
\item[1)]  a slide under the second arc $\sigma_2\sigma_1^2\sigma_2$;
\item[2)]  a slide around the $j$-th longitude  $a_j\sigma_1^{-1}a_j\sigma_1^{-1}$;
\item[3)]  a slide around the $j$-th meridian $b_j\sigma_1^{-1}b_j\sigma_1^{-1}$.
\end{itemize}
\end{itemize}

\begin{figure}[h!]
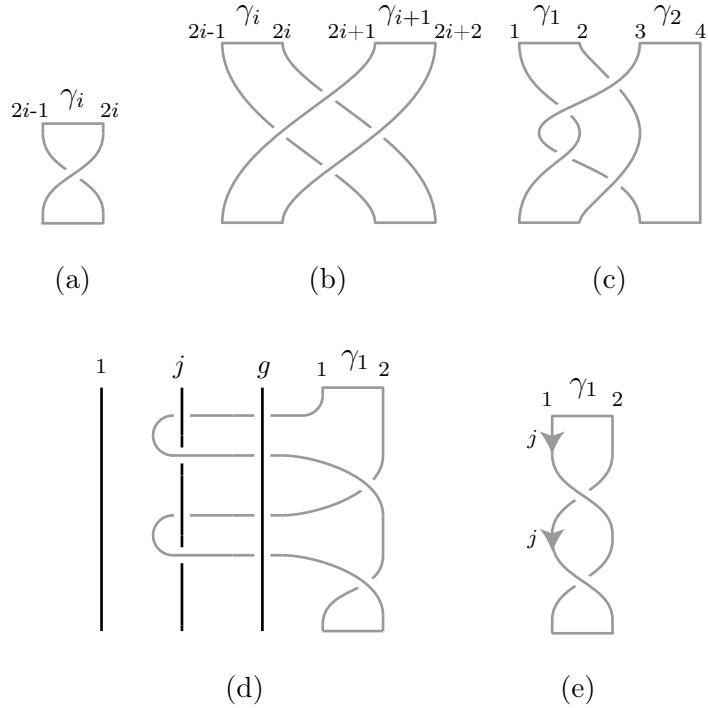

\centering

\begin{subfigure}[b]{.2\linewidth}\centering
\begin{overpic}[page=66]{images}
\put(-15,83){\msss{2i\mm 1}}\put(45,83){\msss{2i}}\put(18,93){$\gamma_i$}
\end{overpic}\caption{}\end{subfigure}
\
\begin{subfigure}[b]{.25\linewidth}\centering
\begin{overpic}[page=67]{images}
\put(-10,84){\msss{2i\mm 1}}\put(25,84){\msss{2i}}\put(10,92){$\gamma_i$}
\put(48,84){\msss{2i\pp 1}}\put(92,84){\msss{2i\pp 2}}\put(69,92){$\gamma_{i+1}$}
\end{overpic}\caption{}\end{subfigure}
\
\begin{subfigure}[b]{.25\linewidth}\centering
\begin{overpic}[page=68]{images}
\put(0,84){\msss{1}}\put(28,84){\msss{2}}\put(10,92){$\gamma_1$}
\put(52,84){\msss{3}}\put(77,84){\msss{4}}\put(60,92){$\gamma_2$}
\end{overpic}\caption{}\end{subfigure}\\[2em]
\
\begin{subfigure}[b]{.4\linewidth}\centering
\begin{overpic}[page=69]{images}
\put(2,89){\msss{1}}\put(27,89){\mss{j}}\put(54,89){\mss{g}}
\put(73,88){\msss{1}}\put(93,88){\msss{2}}
\put(81,92){$\gamma_1$}
\end{overpic}\caption{}\end{subfigure}
\
\begin{subfigure}[b]{.2\linewidth}\centering
\begin{overpic}[page=70]{images}
\put(2,87){\msss{1}}\put(27,87){\msss{2}}\put(12,92){$\gamma_1$}\put(-3,75){$_j$}\put(-3,39){$_j$}
\end{overpic}\caption{}\end{subfigure}

   \caption{(a) The braid twist of the $i$-th arc, (b) the elementary exchange of the $i$-th and $(i+1)$-th arc, (c) the slide of the fist arc under the second  one, (d) the slide of the fist arc around the $j$-th longitude, (e) the slide of the fist arc around the $j$-th meridian.}
\label{fig:f}
\end{figure}
 Note that the slide of the first arc under the $i$-th one is the product of $d\sigma_2\sigma_1^2\sigma_2d^{-1}$, where $d$ is an elementary exchange taking the $i$-th arc into the second  one. 

\begin{rem}
 Let $\textup{MCG}_{2n}(\Sigma_g):=\pi_0(\textup{Homeo}^+(\Sigma_g,\mathcal P_{2n}))$ and $\textup{MCG}(\Sigma_g)=\pi_0(\textup{Homeo}^+(\Sigma_g))$ and consider the exact sequence (see \cite{B1})

$$\cdots \to B_{g,2n}\to\textup{MCG}_{2n}(\Sigma_g)\to   \textup{MCG}(\Sigma_g)\to 1.$$

The image of the above defined  elements of $B_{g,2n}$   belong to the Hilden braid group $\textup{Hil}^g_n\subset \textup{MCG}_{2n}(\Sigma_g)$ introduced in \cite{BC}. Such subgroup can be characterized as that containing the  elements admitting an extension to the couple $(H,\mathcal A)$, where $H$ is the handlebody corresponding to the system of curves depicted in  Figure~\ref{figstandard} and $\mathcal A$ is a system of trivial arcs properly embedded in $H$ and  projecting onto $\{\gamma_1,\ldots, \gamma_n\}$.  
\end{rem}

We are ready to state the main theorem of this section.

\begin{teo}\label{thm:markov-sigma}
 Two elements of $\cup_{n\in\mathbb N} B_{g,2n}$ determine, via plat closure, equivalent links in $\Sigma_g\times I$ if and only if they are connected by a finite sequence of the following moves:
\begin{align*}
  (M1)  & \qquad \sigma_1\beta \longleftrightarrow  \beta  \longleftrightarrow \beta \sigma_1\\
  (M2)  & \qquad \sigma_{2i}\sigma_{2i+1}\sigma_{2i-1}\sigma_{2i} \beta \longleftrightarrow \beta \longleftrightarrow \beta \sigma_{2i}\sigma_{2i+1}\sigma_{2i-1}\sigma_{2i} \\
  (M3)  & \qquad \sigma_2\sigma_1^2\sigma_2\beta \longleftrightarrow  \beta  \longleftrightarrow \beta \sigma_2\sigma_1^2\sigma_2\\
   (M4)  & \qquad a_j\sigma_1^{-1}a_j\sigma_1^{-1}\beta \longleftrightarrow \beta\longleftrightarrow \beta a_j\sigma_1^{-1}a_j\sigma_1^{-1}\quad \textup{for } j=1,\ldots, g \\
  (M5) &\qquad b_j\sigma_1^{-1}b_j\sigma_1^{-1}\beta \longleftrightarrow \beta\longleftrightarrow \beta b_j\sigma_1^{-1}b_j\sigma_1^{-1}\quad \textup{for } j=1,\ldots, g \\
  (M6) &\qquad \beta \longleftrightarrow T_k(\beta)\sigma_{2k}\\
  \ &\ \ \quad \textup{ where } T_k: B_{g,2n}\to B_{g, 2n+2} \textup{ is  defined by } 
   T_k(a_i)=a_i, \ T_k(b_i)=b_i \textup{ and }\\
\ & \ \ \quad T_k(\sigma_i)=\left\{\begin{array}{l}\sigma_i \qquad \qquad \qquad \qquad \qquad \qquad\quad\textup{if } i<2k\\
                          \sigma_{2k}\sigma_{2k+1}\sigma_{2k+2}\sigma_{2k+1}^{-1}\sigma_{2k}^{-1}\ \qquad \qquad \textup{if } i=2k\\
                          \sigma_{i+2} \qquad \qquad \qquad \qquad \qquad \qquad \textup{if } i>2k\\
                         \end{array}\right..
  \end{align*}
\end{teo}


\begin{proof}
Following the proof of Theorem  \ref{lemma:lins}, given   $L$  in standard position in order to find an element $\beta\in B_{g,2n}$ such that $\widehat \beta=L$, up to isotopy, we have to choose two paths $p^-$ and $p^+$ in $C_{n}(\Sigma_g)$ connecting, respectively, the set of lower boundary  points   of $L$ and the set of upper boundary points of $L$ to $\{B_1,\ldots, B_n\}$, where $B_i$ is an internal point of the arc $\gamma_i$, with $i=1,\ldots,n$. Clearly the choice of such paths is not unique: however if $p^-$ and $q^-$ are two possible choices for the set of lower boundary points, the composition of $p^-$ with the inverse of $q^-$ determines an element in $\pi_1(C_n(\Sigma_g), \{B_1,\ldots, B_n\})$ that could be realized as a composition of  braid twists, exchanges and slides  of the arcs $\gamma_i$. An analogous remark holds for upper boundary points. So two different elements of $B_{g,2n}$ obtained as above for the same link $L$ in standard position are connected by moves $(M1),\ldots, (M5)$.

By  Lemma~\ref{lemma_birman},  in order to prove the statement,  it is enough to describe how  a stabilization move and a spike move change a braid representative $\beta$  of a link $\widehat\beta$.

We start with the stabilization move. The effect of a stabilization move is to add a  trivial loop to the plat $\widehat\beta$ at any point. As depicted in Figure~\ref{fig:1}, by  ``sliding the stabilization'' along a connected component,  we may assume, up to spike moves, that the  stabilization is done at the bottom right of an even strand. We have four different possibilities on how to perform the stabilization move depending whether (i) we add a loop with a positive or negative twist and  (ii)  the new strands pass in front or behind the old ones. Up to spike moves, it is always possible to assume that  the twist is positive (see Figure~\ref{figstabnormd}) and  that the new strands pass in front of the old ones (see Figure~\ref{fig:3}). With these assumptions, it is straightforward to check that    the braid representative after the stabilization move will be $T_k(\beta)\sigma_{2k}$ if the stabilization occurs on the $2k$-th strand,   that is the stabilization move corresponds  to a $(M6)$ move.

\begin{figure}[ht]
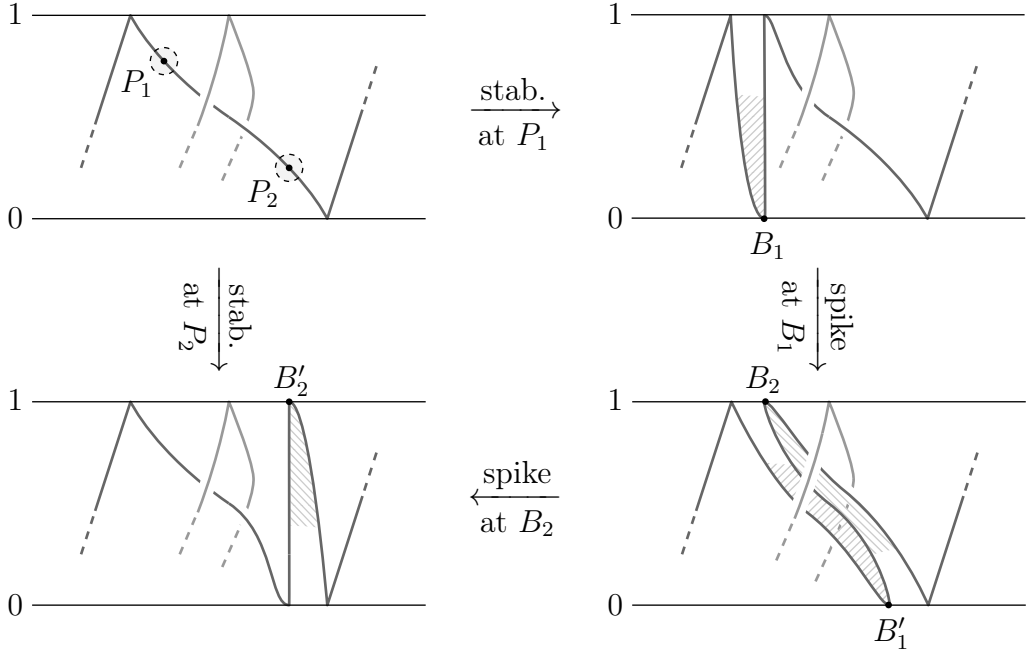

\centering
\begin{tabular}{ccc}
    \begin{overpic}[page=50]{images}
    \put(3,51.5){$1$}\put(3,5){$0$}\put(28.5,36){$P_1$}\put(57,11){$P_2$} 
    \end{overpic} &
    \raisebox{1.75cm}{$\xrightarrow[\mbox{at $P_1$}]{\makebox[1.0cm]{stab.}}$} &
    \begin{overpic}[page=51]{images}\put(3,51.5){$1$}\put(3,5){$0$}\put(35,-1){$B_1$}
    \end{overpic} \\
    \raisebox{1.5em}{\rotatebox{-90}{$\xrightarrow[\mbox{at $P_2$}]{\makebox[1.2cm]{stab.}}$}} &
     & \raisebox{1.5em}{\rotatebox{-90}{$\xrightarrow[\mbox{at $B_1$}]{\makebox[1.2cm]{spike}}$}}\\
    \begin{overpic}[page=52]{images}
    \put(3,51.5){$1$}\put(3,5){$0$}\put(63,56.5){$B_2'$}
    \end{overpic} &
    \raisebox{1.75cm}{$\xleftarrow[\mbox{at $B_2$}]{\makebox[1.0cm]{spike}}$} &
    \begin{overpic}[page=53]{images}\put(3,51.5){$1$}\put(3,5){$0$}\put(63,-1){$B_1'$}\put(34,56.5){$B_2$}
    \end{overpic}
\end{tabular}
        \caption{Normalizing the stabilization: always done  at the bottom right of an even strand.}
\label{fig:1}
\end{figure}

\begin{figure}[ht]
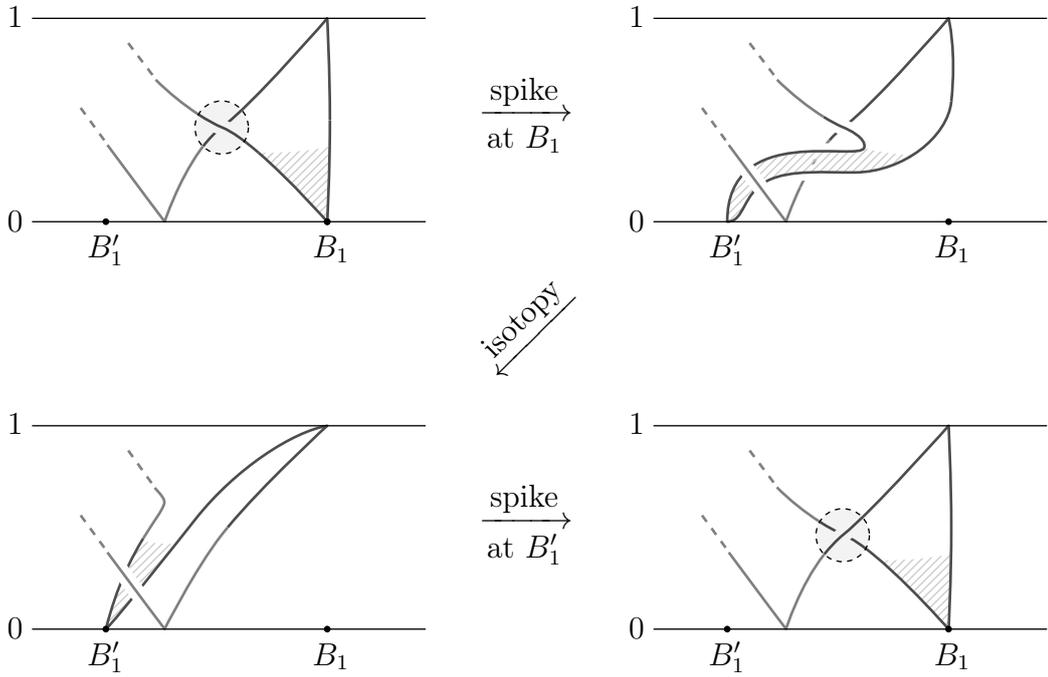

\centering
\begin{tabular}{ccc}
    \begin{overpic}[page=54]{images}
    \put(3,51.5){$1$}\put(3,5){$0$}\put(21,-1){$B_1'$}\put(72,-1){$B_1$}
    \end{overpic} &
    \raisebox{1.75cm}{$\xrightarrow[\mbox{at $B_1$}]{\makebox[1.0cm]{spike}}$} &
    \begin{overpic}[page=55]{images}\put(3,51.5){$1$}\put(3,5){$0$}\put(21,-1){$B_1'$}\put(72,-1){$B_1$}
    \end{overpic} \\
    & \rotatebox{45}{$\xleftarrow{\makebox[1.3cm]{isotopy}}$}& \\
    \begin{overpic}[page=56]{images}
    \put(3,51.5){$1$}\put(3,5){$0$}\put(21,-1){$B_1'$}\put(72,-1){$B_1$}
    \end{overpic} &
    \raisebox{1.75cm}{$\xrightarrow[\mbox{at $B_1'$}]{\makebox[1.0cm]{spike}}$} &
    \begin{overpic}[page=57]{images}\put(3,51.5){$1$}\put(3,5){$0$}\put(21,-1){$B_1'$}\put(72,-1){$B_1$}
    \end{overpic}
\end{tabular}
        \caption{Normalizing the stabilization: adding always a positive twist.}
\label{figstabnormd}
\end{figure}

\begin{figure}[ht]
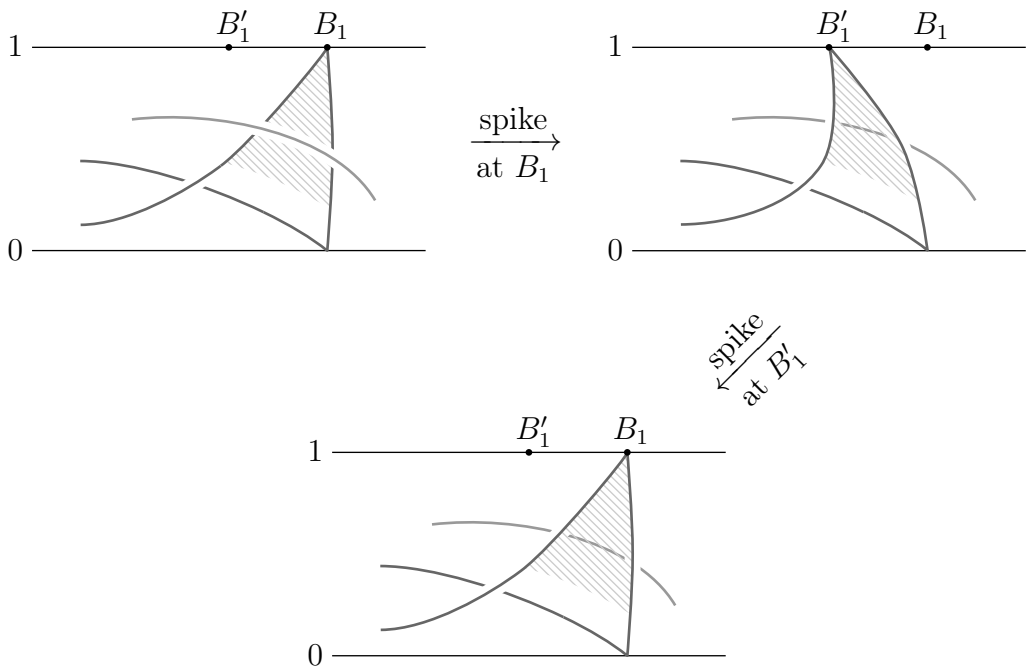

\centering
\begin{tabular}{ccc}
    \begin{overpic}[page=58]{images}
    \put(3,51.5){$1$}\put(3,5){$0$}\put(50,56.5){$B_1'$}\put(72,56.5){$B_1$}
    \end{overpic} &
    \raisebox{1.75cm}{$\xrightarrow[\mbox{at $B_1$}]{\makebox[1.0cm]{spike}}$} &
    \begin{overpic}[page=59]{images}\put(3,51.5){$1$}\put(3,5){$0$}\put(50,56.5){$B_1'$}\put(72,56.5){$B_1$}
    \end{overpic} \\ &
    & \lapbox[\width]{-2em}{\rotatebox{45}{$\xleftarrow[\mbox{at $B_1'$}]{\makebox[1.0cm]{spike}}$}} \\
    \multicolumn{3}{c}{
    \begin{overpic}[page=60]{images}    
    \put(3,51.5){$1$}\put(3,5){$0$}\put(50,56.5){$B_1'$}\put(72,56.5){$B_1$}
    \end{overpic}  }
\end{tabular}
         \caption{Normalizing the stabilization: the new strands pass always in front of the old ones.}
\label{fig:3}
\end{figure}

Let $\widehat\beta'$ be the plat obtained from $\widehat\beta$ by applying a spike move. Suppose that the spike move involves the $i$-th upper  boundary  point $B^+_i$, with $i\in\{1,\ldots,n\}$. We may assume that there exists  $t_0\in (0,1)$ such that the  segments  $[P,P']$ and   $[Q,Q']$ lie in $\Sigma_g\times \{t_0\}$ and that the surface $\Sigma_g\times \{t_0\}$ divides the braid $\beta$ into  an upper braid $\beta_1$ and a lower braid $\beta_2$, both contained in $B_{g,2n}$,  so  that $\beta=\beta_1\beta_2$.  Then there exists an element $\beta_1'\in B_{g,2n}$ such that   $\beta'=\beta'_1\beta_2$. Consider the element $\beta'\beta^{-1}=\beta_1'\beta_1^{-1}$ and denote by $(\beta'\beta^{-1})_{k}$   the $k$-th strand of the braid $\beta'\beta^{-1}$. It follows from the  definition of  spike move that (i) there exists an embedding of a band $I\times I$ into $\Sigma_g\times I$ whose boundary is given by $\gamma_i\times\{1\}\cup(\beta'\beta^{-1})_{2i-1}\cup  \gamma_{j}\times\{0\}\cup(\beta'\beta^{-1})_{2i}$ with  $i\in\{1,\ldots,n\}$ and (ii) $\beta'\beta$ is in the kernel of the map $B_{g,2n}\to B_{g,2n-2}$ obtained by forgetting the $(2i-1)$-th point and $2i$-th point (see Figure~\ref{fig:g}). This means exactly that the element $\beta'\beta^{-1}$ can be written as  a product of slides of arcs,  braid twists and exchanges  of arcs  that correspond to the moves $(M1),\ldots, (M5)$. An analogue reasoning holds in the case of a spike move on a lower boundary point. 

\begin{figure}[h!]
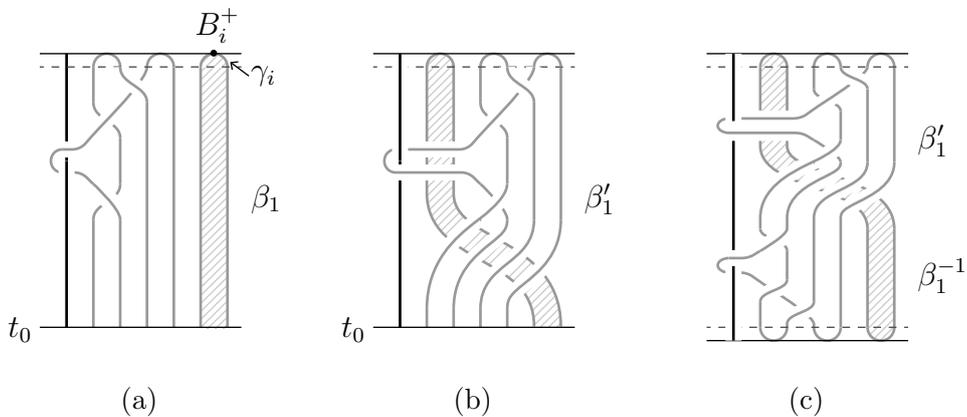

\centering

\begin{subfigure}[b]{.3\linewidth}\centering
\begin{overpic}[page=61]{images}
\put(-4,7){$t_0$}\put(53,100){$B_i^+$}\put(70,86){$\gamma_i$}\put(70,47){$\beta_1$}
\end{overpic}\caption{}\end{subfigure}
\
\begin{subfigure}[b]{.3\linewidth}\centering
\begin{overpic}[page=62]{images}
\put(-4,7){$t_0$}\put(70,47){$\beta_1'$}
\end{overpic}\caption{}\end{subfigure}
\
\begin{subfigure}[b]{.3\linewidth}\centering
\begin{overpic}[page=63]{images}
\put(70,65){$\beta_1'$}\put(70,23){$\beta_1^{-1}$}
\end{overpic}\caption{}\end{subfigure}
\
\caption{Braid interpretation of the spike move.}
\label{fig:g}
\end{figure}
\end{proof}

Given two braids we say that they are \emph{$\Sigma$-equivalent} if it is possible to connect them by a finite sequence of the six \emph{$M$-moves} $(M1),\ldots, (M6)$.

\end{section}

\begin{section}{The slide move and the Markov theorem}
\label{slide}

In this section we establish the main result of the paper,  that is, we describe the moves connecting braids representing isotopic links, by adding slide like moves to $\Sigma$-equivalence. \\

Let $L_1$ and $L_2$ be two disjoint links in $M$ and let $b \cong I\times I$ be an embedded band in $M$, such that
$b \cap L_1  = e_1 \cong I \times \{ 0 \}$ and $b \cap L_2 = e_2  \cong I \times \{ 1 \}$. The \emph{band connected sum} of $L_1$ and $L_2$ along $b$ is the link 
$$(L_1 - e_1) \cup (L_2 - e_2) \cup \overline{(\partial b - e_1 - e_2)},$$
denoted by $L_1 \#_b L_2$.



\begin{rem}\label{remark:sum}

In general the band connected sum $L_1 \#_b L_2$ depends on the choice of the band $b$.
For an oriented split link $L_1 \cup L_2$ in $S^3$, 
 such that the splitting sphere intersects $b$ transversally in a single arc, we can argue by the lightbulb trick (see~\cite{rolfs}) that the band connected sum is independent on the choice of $b$ (up to the choice of components to which $b$ connects).
For a general 3-manifold $M$, a sufficient condition for the band connected sum to be independent of $b$ (up to the choice of the component in $L_1$ to which $b$ connects) is that $L_2$ is an unknot contained inside a 3-ball $B^3$, that is disjoint from $L_1$, and $b$ does not intersect the open disk which $L_2$ bounds. We say that such a band $b$ is \emph{unlinked} with $L_2$.

\end{rem}
Next we define a connected sum operation $\alpha \# \beta$ for braids $\alpha$ and $\beta$ so that it holds $\widehat{\alpha \# \beta} = \widehat{\alpha} \#_b \widehat{\beta}$ for some band $b$.

Let $\alpha \in B_{g,2m}$ and $\beta \in B_{g,2n}$ be two braids.
The \emph{plat sum} of $\alpha$ and $\beta$ is the operation
$$\alpha \# \beta := \alpha\, w_{m,n} \,\beta \; \in B_{g,2(m+n-1)},$$
where $$w_{m,n} = \prod_{i=0}^{2m-3} \prod_{j=0}^{2n-3} \sigma_{2m-i+j},$$
see Figure~\ref{fig:connected_sum} for a geometric interpretation.

\begin{figure}[ht]
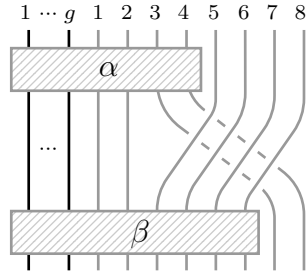

\centering
	\begin{overpic}[page=15]{images}
	 \put(31,64){$\alpha$}\put(41,11){$\beta$}
	 \put(12,39){$\scriptstyle\cdots$}\put(12,82){$\scriptstyle\cdots$}
	 \put(6,82){$\scriptstyle 1$}\put(20,82){$\scriptstyle g$}
	 \put(29,82){$\scriptstyle 1$}\put(38,82){$\scriptstyle 2$}\put(47,82){$\scriptstyle 3$}
	 \put(56,82){$\scriptstyle 4$}\put(65.5,82){$\scriptstyle 5$}\put(75,82){$\scriptstyle 6$}
	 \put(84,82){$\scriptstyle 7$}\put(93,82){$\scriptstyle 8$}
	 \end{overpic}
        \caption{The plat sum $\alpha \# \beta$.}
\label{fig:connected_sum}
\end{figure}

Let again $\Sigma_g$ be a genus $g$ Heegaard surface of $M$ and $\mathbf{c} = \{c_1, \ldots, c_g\}$ the collection of attaching circles in  $\Sigma_g\times \{1\}$ and $\mathbf{ c^*} =\{c^*_1,\ldots, c^*_g\}$ the collection of dual attaching circles in $\Sigma_g\times \{0\}$.

If we approach the attaching region of the $i$-th 2-handle in $\Sigma_g \times \{1\}$  with an arc of a link  $L\subset \Sigma_g\times I$, we can slide the arc along the 2-handle, which has the effect of making a connected sum with the attaching circle $c_i$ by a small band $b$:
$$sl_i: L \lra L \#_b c_i, \quad i = 1, \ldots, g.$$
We call this operation the \emph{$i$-th slide move}.
Similarly, if we approach the attaching region of the dual $i$-th $2$-handle in $\Sigma_g \times \{0\}$, this gives rise to the \emph{$i$-th dual slide move}:
$$sl^*_i: L \lra L \#_b c^*_i, \quad i = 1, \ldots, g.$$

Both types of slide moves are isotopy moves in $M$, since all $c_i$ and $c_i^*$ bound (meridian) discs in $M$ and  thus are trivial knots in $M$.

\begin{lemma}\label{lemma:twostrands}
The curves $\mathbf{c}$ and dual curves $\mathbf{c^*}$ can be expressed as closed plats with two strands.
\end{lemma}
\begin{proof}
Let $Q_1,\ldots, Q_k$ be the collection of consecutive vertices in a PL-decomposition of $c_i$. Since $c_i$ lies on $\Sigma_g \times \{ 1\}$,
all arcs of $c_i$ are horizontal with respect to the height function associated with $\Sigma_g \times I$. By a small perturbation we can isotope the points so that all arcs $[ Q_i, Q_{i+1} ]$ are oriented downwards and use a $\Delta$-move to replace the arc $[Q_k ,Q_1]$ with $[Q_k, Q_1']\cup [Q_1', Q_1]$, where $[Q_1', Q_1]$ is a vertical upward arc. By the braiding process described in the proof of Theorem~\ref{lemma:lins} (see also Figure~\ref{fig_alex}), a knot with only one upward arc can be braided with two strands.
An analogue construction can be made for $c_i^*$.  
\end{proof}

If $\beta \in \B_{g,2n}$ is a braid representative of $L$ and $\overline{c_i}$ (resp. $\overline{c^*_i}$) is a braid representative of $c_i$ (resp. $c_i^*$), then the slide move (resp. dual slide move) can be expressed in braid form by a plat connected sum as:
$$psl_i: \beta \lra \overline{c_i} \# \beta,\quad i = 1, \ldots, g,$$
which we call the \emph{$i$-th plat slide move} and 
$$psl_i^*: \beta \lra \beta \# \overline{c_i^*},\quad i = 1, \ldots, g,$$
which we name the \emph{$i$-th dual plat slide move}.
Since $c_i$ can be braided with two strands, both $psl_i(\beta)$ and $psl_i^*(\beta)$ are elements of $B_{g,2n}$. 

There are several ways we can slide an arc across a 2-handle. Since we are performing band sums with trivial knots, by Remark~\ref{remark:sum}, we have to check that the above plat slide moves include band connected sums where the band starts at any position of $\widehat{\beta}$ and any position of $c_i$ (resp. $c_i^*$), assuming that  the band  $b$ is unlinked with $c_i$ (resp. $c_i^*$). The following lemma shows that, up to $\Sigma$-equivalence, plat slide moves include band connected sums where the band starts at any position of $\widehat{\beta}$. 

\begin{lemma}\label{lemma:top}
Given a plat $\widehat\beta$,  a  plat slide move (resp. dual plat slide move) can be always assumed to take place on top left strand (resp. bottom left strand) of $\beta$  as represented in Figure~\ref{fig:proofcb1}. That is, for any braid $\beta$ and  any band $b$ starting from an arbitrary point in $\widehat \beta$, arbitrary linked with $\widehat \beta$  and connected to $c_i$ (resp. $c_i^*)$,  there exists a $\Sigma$-equivalent braid $\beta'$ such that $  c_i \#_b \widehat\beta $ is isotopic to $ \widehat{\overline{c_i} \# \beta'}$ (resp. $\widehat  \beta \#_b c^*_i $ is isotopic to $\widehat{\beta' \# \overline{c^*_i}}$). 

\end{lemma}

\begin{proof}
We prove the statement in  the non-dual case.
As the band $b$ approaches $c_i$ (Figure~\ref{fig:proofcb2}) we can cut $b$ and obtain a link $\widehat{\beta'}$ isotopic to $\widehat \beta$ (Figure~\ref{fig:proofcb3}). We use the braiding process described  the proof of Theorem~\ref{lemma:lins} (see also Figure~\ref{fig_alex}) and put the new link in plat position with  the braid $\beta'$ being $\Sigma$-equivalent to $\beta$ and having the connecting arc at the top left.

\end{proof}

\begin{figure}[ht]
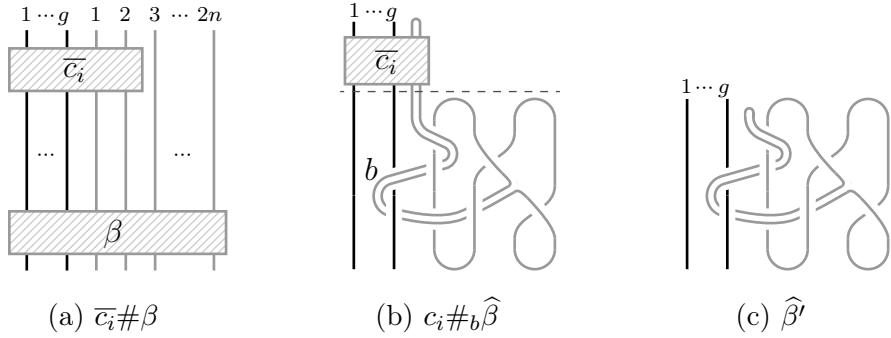

\centering
\begin{subfigure}[b]{0.2\textwidth}    
\centering
	\begin{overpic}[page=20]{images}
	\put(14,44){$\scriptstyle\cdots$}\put(14,96){$\scriptstyle\cdots$}
	 \put(8,96){$\scriptstyle 1$}\put(22,96){$\scriptstyle g$}\put(34,96){$\scriptstyle 1$}\put(45,96){$\scriptstyle 2$}\put(56,96){$\scriptstyle 3$}\put(74,96){$\scriptstyle 2n$}\put(64,96){$\scriptstyle\cdots$}\put(65,44){$\scriptstyle\cdots$}
	 \put(24,75){$\overline{c_i}$}\put(39,13){$\beta$}\end{overpic}	
        \caption{$\overline{c_i} \# \beta$}
        \label{fig:proofcb1}
    \end{subfigure}\qquad\qquad
\begin{subfigure}[b]{0.2\textwidth}    
\centering
	\begin{overpic}[page=21]{images}\put(16,79){$\overline{c_i}$}\put(12,37){$b$}
	\put(12,99){$\scriptstyle\cdots$}\put(6,99){$\scriptstyle 1$}\put(20,99){$\scriptstyle g$}
	 \end{overpic}	
        \caption{$c_i \#_b \widehat{\beta}$}
        \label{fig:proofcb2}
    \end{subfigure}\qquad\qquad
\begin{subfigure}[b]{0.2\textwidth}    
\centering
	\begin{overpic}[page=22]{images}
	 \put(12,79){$\scriptstyle\cdots$}\put(6,79){$\scriptstyle 1$}\put(22,79){$\scriptstyle g$}
	 	 \end{overpic}	
        \caption{$\widehat{\beta'}$}
        \label{fig:proofcb3}
    \end{subfigure}
\caption{Normalizing the plat slide move.}
\label{fig:proofcb}
\end{figure}



Next lemma shows that, up to $\Sigma$-equivalence,  the plat slide moves $psl_i$ (resp. $psl_i^*$) do  not depend on the point where the band $b$  is attached   to $c_i$ (resp. $c_i^*$).

\begin{lemma}\label{lemma:c}



Let $c_i \#_b \widehat{\beta}$ (resp. $\widehat{\beta} \#_b c_i^*$) be the band connected sum, where $b$ is unlinked with $c_i$ (resp. $c_i^*$) and connected to $c_i$ (resp. $c_i^*$) at a small arc $e_1$. Let $e_1'$ be another small arc on $c_i$ (resp. $c_i^*$), then there exists a braid $\beta'$ which is $\Sigma$-equivalent to $\beta$ and a band $b'$ which is unlinked with $c_i$ (resp. $c_i^*$) and connected to $c_i$ (resp. $c_i^*$) at $e_1'$, such that $c_i \#_b \widehat{\beta}$ is isotopic to $ c_i \#_{b'} \widehat{\beta'}$ (resp. $\widehat{\beta} \#_b  c_i^*$ is isotopic to $\widehat{\beta'} \#_{b'} c_i^*$).

\end{lemma}
\begin{proof}We prove the statement in the non-dual case.
By Lemma~\ref{lemma:top} we can assume that $b$ is connected with the top-left arc of $\beta$.
Observe that we can slide the small arc $e_1$ together with the band $b$ along the knot $c_i$ towards $e_1'$ (see Figures~\ref{fig:proofc1} and~\ref{fig:proofc2} for the case $g=1$). In this process it can happen that the connecting band crosses a lateral surface of $G \times I$, where $G$ is the fundamental polygon of $\Sigma_g$. If we keep track of the braiding process before and after that the band crosses a lateral surface, we see that the two braids differ by either an $(M4)$ or an $(M5)$ move (see Figure~\ref{fig:proofc2}). 
When we reach $e_1'$, we can push the braiding of the band to the braid and the entire process gives rise to a braid $\beta'$ (see Figure~\ref{fig:proofc3}),
 which is $\Sigma$-equivalent to $\beta$ so it holds that $c_i \#_{b'} \widehat{\beta'}$ is isotopic to the original connected sum $c_i \#_b \widehat{\beta}$.


\end{proof}

\begin{figure}[ht]
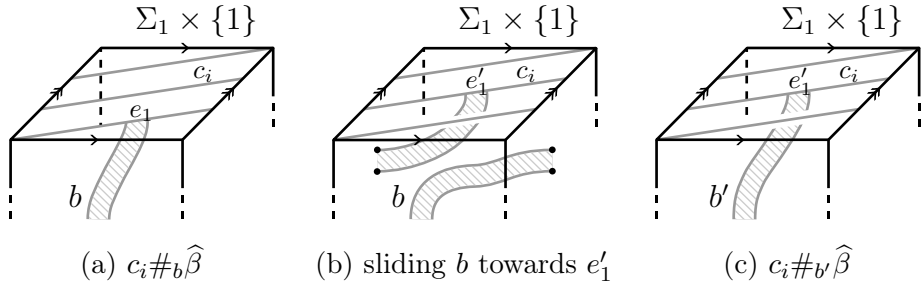

\centering
\begin{subfigure}[b]{0.3\textwidth}    
\centering
	\begin{overpic}[page=17]{images}\put(68,51){\mss{c_i}}\put(46,37.5){\mss{e_1}}
	\put(48,68){$\Sigma_1\times\{1\}$}\put(25,6){$b$}
	\end{overpic}	
        \caption{$c_i \#_b \widehat{\beta}$}
        \label{fig:proofc1}
    \end{subfigure}
   \begin{subfigure}[b]{0.3\textwidth}  
\centering
	\begin{overpic}[page=18]{images}\put(68,51){\mss{c_i}}\put(50.5,47.5){\mss{e_1'}}
	\put(48,68){$\Sigma_1\times\{1\}$}\put(25,6){$b$}
	 	 \end{overpic}	
        \caption{sliding $b$ towards $e_1'$}
        \label{fig:proofc2}
    \end{subfigure}
   \begin{subfigure}[b]{0.3\textwidth}    
\centering
	\begin{overpic}[page=19]{images}\put(68,51){\mss{c_i}}\put(50.5,47.5){\mss{e_1'}}
	\put(48,68){$\Sigma_1\times\{1\}$}\put(23,6){$b'$}
	\end{overpic}	
        \caption{$c_i \#_{b'} \widehat{\beta}$}
        \label{fig:proofc3}
    \end{subfigure}

\caption{Normalizing the plat slide move.}
\label{fig:proofc}
\end{figure}

We are ready to prove the main theorem of this paper.

\begin{teo}[Markov theorem]\label{thm:markov}
Let $L_1$ and $L_2$ be two links in $M$ such that $\widehat \beta_i=L_i$  with $\beta_i\in \cup_{n\in\mathbb N} B_{g,2n}$, with $i=1,2$. Then $L_1$ and $L_2$ are isotopic if and only if $\beta_1$ and $\beta_2$ differ by
a finite sequence of
braid isotopy moves~($R1$), $\dotso$, ($R4$), ($TR$); 
$M$-moves ($M1$), $\dotso$, ($M6$);
plat slide moves $psl_i$, $i=1,\ldots,g$ and dual plat slide moves $psl^*_i$, $i=1,\ldots,g$.
\end{teo}
\begin{proof}
An isotopy between  two links in $M$ can be obtained by an isotopy in $\Sigma_g \times I$ with the additional freedom to slide across the 2-handles described by $\mathbf{c}$ and the dual 2-handles described by $\mathbf{c^*}$ (sliding across 0-handles and 3-handles is of course trivial). Isotopy is thus, by Theorem~\ref{thm:markov-sigma}, described by the $M$-moves and slide moves along meridian discs. With Lemmas~\ref{lemma:top} and~\ref{lemma:c} we have shown that it is enough to consider only the plat slide move $psl_i$  for each 2-handle and the dual plat slide move $psl^*_i$ for each dual 2-handle. 
\end{proof}









If we assume that $\mathbf c^*$ is the system corresponding to the  curves depicted in Figure~\ref{figstandard}, then  we have $c_i^*=\widehat b_i$. In this case, as the following proposition shows,   the $b$-type generators are redundant  in order to describe a link in $M$  as the plat closure of a braid $\beta \in B_{g,2n}$.

\begin{prop}\label{prop:no_b}
A braid $\beta \in B_{g,2n}$  is $\Sigma$-equivalent to a braid  $\beta' \in B_{g,2n}$  with no $b$-type generators. Furthermore, all $b$-type generators can be removed using $M$-moves and $sl_i^*$ moves. 
\end{prop}
\begin{proof}



First observe that we can push a $b$-type generator through an $a$-type generator using either $(R4)$ or $(R3)$. Pick the last $b$ generator, i.e. a letter $b_i\pmo$, in the word $\beta$ (Figure~\ref{fig:b1}) and make a stabilization move right after the generator (Figure~\ref{fig:b2}). For each $a$-type generator we choose the stabilization strands to go either under or over the interfering strands of the $a$-type generator in such a way that relations $(R3)$ and $(R4)$ can be applied. We can now push $b$ to the bottom and remove it by the $sl_i^*$ move (Figures \ref{fig:b3} and \ref{fig:b4}).
We repeat this process until all $b$-type generators are removed (Figures \ref{fig:b5} and \ref{fig:b6})

\begin{figure}[ht]
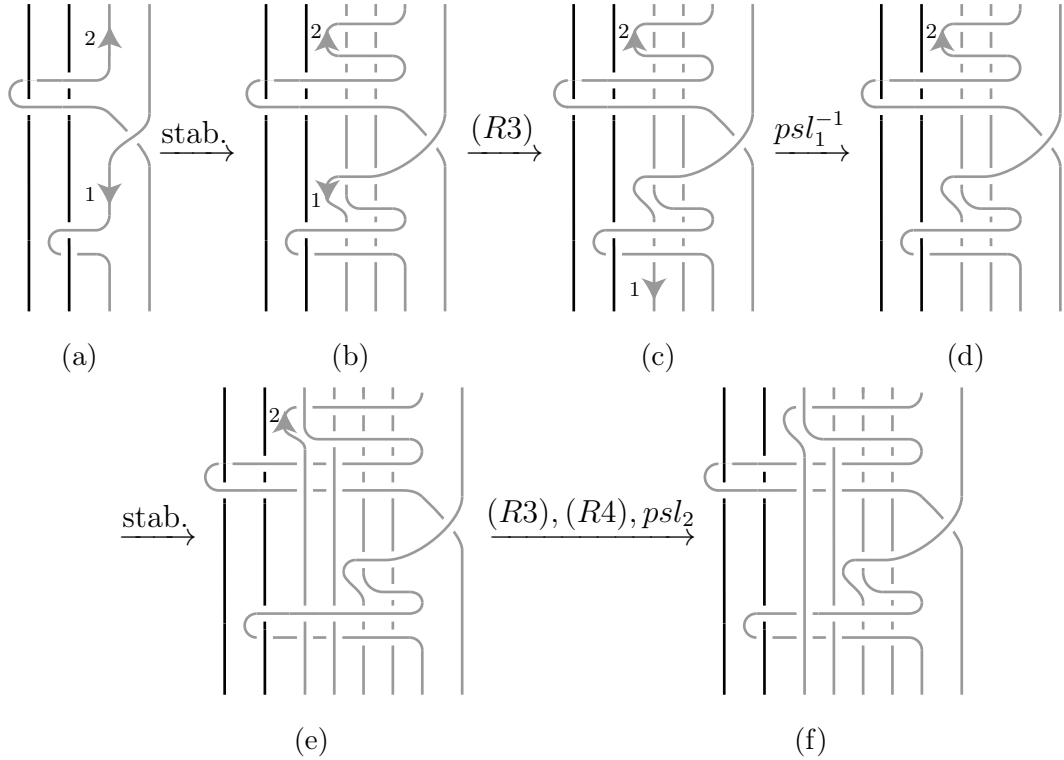

\centering
\begin{subfigure}[b]{0.15\textwidth}    
\centering\begin{overpic}[page=72]{images}\put(26,87){$_2$}\put(26,38){$_1$}\end{overpic}\caption{}\label{fig:b1}
\end{subfigure}\raisebox{7em}{$\xrightarrow[]{\makebox[0.8cm]{stab.}}$}
\begin{subfigure}[b]{0.22\textwidth}    
\centering\begin{overpic}[page=73]{images}\put(22.5,89){$_2$}\put(22.5,36){$_1$}\end{overpic}\caption{}\label{fig:b2}
\end{subfigure}\raisebox{7em}{$\xrightarrow[]{\makebox[0.8cm]{$(R3)$}}$}
\begin{subfigure}[b]{0.22\textwidth}    
\centering\begin{overpic}[page=74]{images}\put(22.5,89){$_2$}\put(26,8){$_1$}\end{overpic}\caption{}\label{fig:b3}
\end{subfigure}\raisebox{7em}{$\xrightarrow[]{\makebox[0.8cm]{$psl_1^{-1}$}}$}
\begin{subfigure}[b]{0.22\textwidth}    
\centering\begin{overpic}[page=75]{images}\put(22.5,89){$_2$}\end{overpic}\caption{}\label{fig:b4}
\end{subfigure}\qquad\raisebox{7em}{$\xrightarrow[]{\makebox[0.8cm]{stab.}}$}
\begin{subfigure}[b]{0.22\textwidth}    
\centering\begin{overpic}[page=76]{images}\put(22.5,89){$_2$}\end{overpic}\caption{}\label{fig:b5}
\end{subfigure}\qquad\raisebox{7em}{$\xrightarrow[]{\makebox[2.5cm]{$(R3),(R4),psl_2$}}$}
\begin{subfigure}[b]{0.22\textwidth}    
\centering\begin{overpic}[page=77]{images}\end{overpic}\caption{}\label{fig:b6}
\end{subfigure}
\caption{Removing $b$-type generators from a braid.}\label{fig:b}
\end{figure}

\end{proof}

\end{section}

\begin{section}{The case of genus one Heegaard  splittings}
\label{example}
In this section we  provide explicit examples of slide moves for manifolds admitting genus one Heegaard splittings, that is lens spaces, $S^2 \times S^1$ and  the 3-sphere.\\

If $M$ is the lens space $L(p,q)$, where $p$ and $q$ are coprime integers such that $0 < q < p$, then
$M$ has a genus 1 Heegaard splitting, sending the meridian of $H_2$ to the  $(p,-q)$-curve on the torus $T=\partial H_1$, see  Figure~\ref{fig:lpq1} for the $L(5,2)$ case.

\begin{figure}[ht]
\centering
\begin{subfigure}[t]{0.4\textwidth}\centering
    \includegraphics*[page=80]{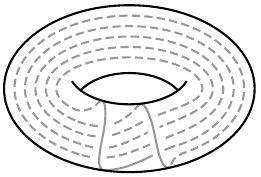}\caption{The $(5,-2)$-knot on $T$.}\label{fig:lpq1}
\end{subfigure}
\begin{subfigure}[t]{0.4\textwidth}\centering
\begin{overpic}[page=79]{images}\put(11,87){$_1$}\put(11,36){$_1$}\end{overpic}\caption{The braid $\alpha_{5,2}$.}\label{fig:lpq2}
\end{subfigure}
\caption{Braiding of the torus knot $(5,-2)$.}
\end{figure}

The  $(p,-q)$ torus knot is the plat closure of the braid with $q$ generators $b^{-1}_1$ evenly distributed between $p$ generators $a_1$ (see also~\cite{G,GM1}):
$$\alpha_{p,q} = \prod_{i=1}^{r} b_1^{-1} a_1^{\lceil \frac{p}{q}\rceil} \cdot \prod_{i=1}^{q-r} b_1^{-1} a_1^{\lfloor \frac{p}{q}\rfloor} \in B_{1,2},$$
where $r \equiv  p \; \mbox{(mod }q\mbox{)}$, see Figure~\ref{fig:lpq2} for the $L(5,2)$ case.

The  plat slide move (resp. dual plat slide move) for $L(p,q)$, which we denote by $psl_{p,q}$ (resp. $psl^*_{p,q}$) is thus
\begin{align*}
psl_{p,q}:\; & \beta \lra \alpha_{p,q} \# \beta = \alpha_{p,q} \, \beta,\\
psl^*_{p,q}:\; & \beta \lra \beta \# b_1 = \beta \, b_1.
\end{align*}

In particular, for $L(p,1)$ we have
$$psl_{p,1}: \beta \lra b_1^{-1}a_1^p \# \beta = b_1^{-1}a_1^p \beta.$$

Beside lens spaces, the only other Heegaard genus one manifold is $S^2\times S^1$, which can be viewed as the degenerate lens space $L(0,1)$. The manifold $S^2\times S^1$ admits a Heegaard splitting $(H_1,H_2,h)$, where $h:\partial H_2 \lra \partial H_1$  sends the meridian  of $H_2$ to the meridian  of $H_1$. 

The plat slide moves in this case are:
\begin{align*}
psl_{0,1}:\; & \beta \lra b_1 \# \beta = b_1 \, \beta,\\
psl^*_{0,1}:\; & \beta \lra \beta \# b_1 = \beta \, b_1.
\end{align*}



Lastly, the 3-sphere $S^3$, viewed as the degenerate lens space $L(1,0)$, admits a genus 1 Heegaard splitting $(H_1,H_2,h)$, where $h:\partial H_2 \lra \partial H_1$ is a homeomorphism that sends the meridian  of $H_2$ to the longitude  of $H_1$.  We have the following two plat slide moves:
\begin{align*}
psl_{1,0}:\; & \beta \lra a_1 \# \beta = a_1\beta, \\
psl^*_{1,0}:\; & \beta \lra \beta \# b_1 = \beta \, b_1.
\end{align*}

With the same methods of Proposition~\ref{prop:no_b} we can kill all $a_1$ generators from a word $\beta \in B_{1,2n}$ by moving them to the top and applying the plat slide move. Every link $L \subset S^3$ can be thus represented as a closed braid without $a_1$ or $b_1$  generators and the theory collapses to that of the the usual genus zero Heegaard splitting
of $S^3$ introduced in~\cite{B}.


\end{section}

\bigskip

\begin{flushleft}

\vbox{
Alessia~CATTABRIGA\\
Department of Mathematics, University of Bologna\\
Piazza di Porta San Donato 5, 40126 Bologna, ITALY\\
e-mail: \texttt{alessia.cattabriga@unibo.it}}

\bigskip

\vbox{Bo\v{s}tjan~GABROV\v{S}EK\\
Faculty of Mathematics and Physics,  University of Ljubljana\\
Jadranska ulica 19, 1000 Ljubljana, SLOVENIA\\
e-mail: \texttt{bostjan.gabrovsek@fmf.uni-lj.si}}

\end{flushleft}

\end{document}